\documentclass[12pt]{article}

\usepackage{amsfonts}
\usepackage{amsmath, amscd}
\usepackage{amssymb}
\usepackage[marginpar]{todo}
\usepackage{paralist} 

\newtheorem{theorem}{Theorem}[section]
\newtheorem{example}{Example}[section]

\newtheorem{remark}{Remark}[section]
\newtheorem{definition}{Definition}[section]

\newtheorem{case}{Case}
\newenvironment{proof}{\textbf{Proof.}}{\qquad $\Box$ \bigskip }

\newcommand{\npmatrix}[1]{\left( \begin{matrix} #1 \end{matrix} \right)} 
\newcommand{\mc}{\mathcal}
\newcommand{\R}{\mathbb{R}}
\newcommand{\C}{\mathbb{C}}
\newcommand{\F}{\mathbb{F}}
\newcommand{\vc}{\mathrm{vec}}
\newcommand{\rk}{\mathrm{rank}}
\newcommand{\pmx}[1]{\ensuremath{\begin{pmatrix}#1\end{pmatrix}}}
\newcommand{\pde}[2]{\frac{ \partial #1}{ \partial #2}}

\title{ Solution Theory for Systems of Bilinear Equations}
\author{Charles R. Johnson$^{1}$, Helena \v Smigoc$^{2}$\thanks{The work of this author was supported by Science Foundation Ireland under Grant 11/RFP.1/MTH/3157} , Dian Yang$^{3}$ \\
{\small $^{1}$Department of Mathematics, The College of William and Mary,}\\ {\small Williamsburg, Virginia, USA}\\
{\small Email:  crjohnso@math.wm.edu} \\
{\small $^{2}$School of Mathematical Sciences,} \\ {\small University College Dublin,} \\ {\small Belfield ,  Dublin 4, Ireland}\\
{\small (Correspondence) Email: Helena.Smigoc@ucd.ie }\\
{\small $^{3}$Department of Mathematics, The College of William and Mary,}\\ {\small Williamsburg, Virginia, USA}\\
{\small Email:  dyang@email.wm.edu}}

\date{}

\begin{document}

\maketitle

\begin{abstract}
Bilinear systems of equations are defined, motivated and analyzed for solvability. Elementary structure is mentioned and it is shown that all solutions may be obtained as rank one completions of a linear matrix polynomial derived from elementary operations. This idea is used to identify bilinear systems that are solvable for all right hand sides and to understand solvability when the number of equations is large or small. 


\parindent=14 pt

\end{abstract}

\section{Introduction}

A \emph{ bilinear equation } is one in which the variables may be partitioned into two (disjoint) subsets such that the left hand side is linear in each set separately. The right hand side is a given scalar. Thus, a bilinear equation may be written as $$y^TAx=g$$
in which $A$ is a $p$-by-$q$ matrix of given scalars from a field $\F$, $A \in M_{p,\, q}(\F),$ $x$ varies over $\F^q,$ while $y$ varies independently over $\F^p,$ and $g \in \F$ is a scalar. A \emph{system of bilinear equations (BLS)} is then 
 $$y^TA_ix=g_i, \, i=1,2\ldots,m$$
with each $A_i \in M_{p, \, q}(\F),$ with $y \in \F^p$ and $x \in \F^q$ two independent vectors of variables (the same for all the equations), and each $g_i \in \F.$ The BLS is called homogeneous if $g_i=0$, $i=1,2,\ldots, m.$
When the $A_i$'s are linearly independent, $m \leq pq,$ and we define $r=pq-m.$  We will fix and use this notation $(p,q,r,m,x,y)$ throughout. 	

Bilinear systems may arise in many ways, but the first author has been motivated to study them because of their connection with the analysis of whether two patterns $\mathcal{P}$ and $\mathcal{Q}$ commute \cite {MR2202432}. In $\mathcal{P}$ and $\mathcal{Q}$ the positions of the nonzero real entries are known but their numerical values are not. The two patterns are said to (real) commute if there exist real matrices $P$ and $Q$, respectively, of patterns $\mathcal{P}$ and $\mathcal{Q}$, that commute. Of course the equation $\mathcal{PQ-QP}=0$ is a special homogeneous BLS in as many variables as the total number of nonzero entries of $\mathcal{P}$ and $\mathcal{Q}$. Here, a totally nonzero solution is required. 
In a recent paper \cite{bobpalais} it was proposed to use pairs of vectors in $\mathbb{R}^3$ to represent quaternions. For $v, w \in \mathbb{R}^3$, formula $T(v,w)=[v\cdot w, v\times w]$ gives a ring isomorphism between equivalence classes of pairs of vectors in $\mathbb{R}^3$ and quaterions. The determination of each equivalence class involves solving the equation $T(v,w)=d_0\in \R^4$, which is a BLS.

Bilinear systems can be connected to bi-affine and multivariate quadratic systems appearing in cryptography, for which some theory and solution algorithms based on linearization and Gr\"obner bases already exist \cite{Courtois02cryptanalysisof,Xie:2009:QEK:1698100.1698124,Wolf}.

A \emph{solution} to  a BLS is a pair of vectors $x \in \F^q,$ $y \in \F^p$ simultaneously satisfying all $m$ bilinear equations. Of course, if $x,$ $y$ is a solution, so is $t x,$ $\frac{1}{t}y$ for all $0 \neq t \in \F$.  The homogeneous BLS always has the "trivial" solutions $x=0,$ $y$ arbitrary, and $x$ arbitrary, $y=0.$ It is natural to focus upon nontrivial solutions in which both $x$ and $y$ have nonzero entries. Our purpose is to develop solution theory for bilinear systems, both to determine whether a given system has any solutions and how to find solutions (there may be many that are essentially different) when they exist. It is worth noting that this is generally difficult (though we develop theory that helps); there are some important likenesses to the theory of linear systems and also some very significant differences. For example, unlike linear systems, there may be no solution over the field in which data is given, while there is a solution over an extension field (rather more like quadratic systems in this regard). 

 \begin{example}
 Let $A_1=\npmatrix{1 & 0 \\ 0 & 1}$ and $A_2=\npmatrix{0 & 1 \\ -1 & 0}.$ The homogenous bilinear system defined by $A_1$ and $A_2:$
\begin{align*}
x_1y_1+x_2y_2&=0 \\
-x_1y_2+y_1x_2&=0
\end{align*} 
has no nontrivial solutions over $\mathbb{R},$ but $y^T=(i,1),$ $x^T=(i,1)$ is a nontrivial solution over $\mathbb{C}.$
 \end{example}

The solvability of bilinear systems seems not to be so well developed thus far. The set of notes \cite{892611} gives several observations, including one (mentioned in Section \ref{Section2}) that is used for an approximate solution algorithm, and \cite{MR2567143} gives a full solution theory (that includes a calculation mechanism in exact arithmetic via linear systems) for complete ($m=pq$ and linearly independent $A_i$'s) bilinear systems that will be a special case of our main result in Section \ref{Section3}.

In the next section, we give some background on BLS's that will be useful throughout and then discuss the simplifying transformations that preserve solutions or solvability of a given BLS. In Section \ref{Section3} we give a general result that transforms both our problems to that of finding rank 1 values of a given linear matrix polynomial (or rank 1 matrices in an affine matrix subspace). We refer to these as  "completions" of the linear matrix polynomial, an interesting question on its own. A brief discussion is given in Section \ref{Section4}. In Section \ref{Section5}, the $A_i,$ $i=1,2,\ldots, m,$ for which the BLS is solvable for all right hand sides (i.e. always solvable) are discussed from several perspectives. In the last section, we discuss solvability of a BLS in other situations. Informative examples are given throughout.

\section{Observations and Solution Preserving Transformations}\label{Section2}

We start with some initial observations that have some overlap with \cite{892611}. The left hand side matrices $A_1,\, A_2,\ldots, A_m$ that define a BLS may be viewed as a $p$-by-$q$-by-$m$ $3$-dimensional array $\mathbb{A}$  when they are displayed side by side. It may be useful to slice this array in other ways. For example, if the vector $y$ (or $x$) is fixed, the bilinear system becomes a linear system $Yx=g$. Here, $g=(g_1,g_2,\ldots,g_m)^T$ and 
 $$Y=\npmatrix{y^TA_1 \\ y^TA_2 \\ \vdots \\ y^TA_m}=y_1R_1+y_2R_2+\ldots+y_pR_p$$
in which $R_i$ is an $m$-by-$q$ matrix and is the slicing of $\mathbb{A}$ with the $i$-th rows of $A_1,\ldots,A_m.$ Similarly, 
 $$X=\npmatrix{A_1x & A_2x & \ldots & A_m x}=x_1S_1+\ldots x_qS_q$$
 in which $S_j$ is a $p$-by-$m$ matrix and has the $j$-th columns of $A_1,\ldots A_m$ in order. If $(x,y)$ is a solution to a bilinear system with right hand side $g,$ then $x$ will be a solution of the linear system $Yx=g$ ($y$ will be a solution to the linear system $y^TX=g^T$), and conversely. 
 
There are two natural types of transformations on the pairs $(A_i, g_i)$ that preserve solvability. The first preserves solutions and is analogous to elementary operations on linear systems:
\begin{enumerate}[(i)]
\item \label{itemi} The $(A_i, g_i)$ pairs may be permuted.
\item \label{itemii} An $(A_i,g_i)$ pair may be multiplied by a nonzero scalar.
\item \label{itemiii} An $(A_i,g_i)$ pair may be replaced by itself plus a linear combination of the other $(A_j,g_j)$ pairs. 
\end{enumerate}
In addition, all the matrices $A_i$ may be simultaneously replaced by a single nonzero scalar multiple of themselves while keeping the right hand sides $g_i$ unchanged. Using operations (\ref{itemi}), (\ref{itemii}), (\ref{itemiii}) described above, the matrices $A_i$
 may be reduced to a (possibly smaller) linearly independent set of matrices $\hat{A}_1,\dots, \hat{A}_{\hat{m}}$ and, possibly, some $0$ matrices, with appropriate modification of the values $g_i$ to the values $\hat{g}_i$ on the right hand side.  If there is a $0$ matrix with nonzero right hand side $\hat{g}_{i},$ then our original bilinear system had no solutions. Otherwise, we may add to our standard assumptions that $A_1,A_2,\ldots, A_m$ are linearly independent, in which case $m \leq pq.$ In addition, $A_1, A_2,\ldots, A_m$  may be taken to be any basis of the space that they span.  

We note that elementary linear operations  (\ref{itemi}), (\ref{itemii}), (\ref{itemiii}) could be used for an additional purpose: to transform $g^T$ to $(1,0,\ldots, 0),$ in the case when the original system is not homogeneous. This was noted in \cite{892611} to transform any bilinear system to an almost homogeneous one. Then solving the bilinear system is like solving a homogeneous one with one fewer equation so that the solution has nonzero bilinear form on another, independent matrix. 

An additional type of transformation that preserves solvability of a bilinear system is simultaneous equivalence on the matrices $A_i$: 
 $$A_i \rightarrow PA_iQ, \, i=1,2,\ldots,m,$$
in which $P \in M_p(\F)$ and $Q \in M_q(\F)$ are nonsingular. In this case, the right hand side $g$
is not changed, and the new bilinear system is solvable if and only if the original one was. In addition, there is a one--to--one correspondence between solutions of the original bilinear system and of the new bilinear system, given by 
 $$(x,y) \rightarrow (Q^{-1}x,P^{-1}y).$$
 It may happen that an equivalence makes the solvability (or non-solvability) of a bilinear system more transparent. 
 
 The transformations on the $A_i$'s  that we have mentioned (linear operations and simultaneous equivalence) may also be interpreted in terms of the matrices  $R_i$ and $S_i.$ For example, right equivalence of the $A_i$'s  ($A_i \rightarrow PA_i,$ $i=1,\ldots,m$) corresponds to right equivalence on the $S_i$'s, and linear elementary operations on the $A_i$'s correspond to left equivalence on the $R_i$'s. 
 
We note that solvability of a bilinear system can depend significantly on the right hand side data.

\begin{example}\label{Ex21}
 Matrices
 $$A_1=\npmatrix{1 & 0 \\ 0 & 1}, \, A_2=\npmatrix{0 & 0 \\ 1 & 0}, \text{ and }A_3=\npmatrix{0 & 1 \\ 0 & 0}$$ 
define the BLS:
 \begin{align*}
 x_1 y_1+x_2y_2=g_1  \\
 x_1 y_2=g_2 \\
 x_2 y_1=g_3.
 \end{align*} 
 If $g_3=0,$ then 
  $$x=\npmatrix{g_2 \\ g_1}, \, y=\npmatrix{0 \\ 1}$$ is a solution and 
 if $g_2=0,$ then 
    $$x=\npmatrix{g_1 \\ g_3}, \, y=\npmatrix{1 \\ 0}$$ is a solution. 
If $g_3 \neq 0$ and $g_2 \neq 0$ then the BLS above has solutions over $\mathbb{R}$ 
if and only if $g_1^2 -4 g_2g_3 \geq 0.$ One can see this by expressing $x_1$ and $x_2$ from the second and the third equation, respectively, and then the first equation becomes a quadratic equation in $y_1/y_2.$    
 \end{example}

\section{General Theory}\label{Section3}

\begin{definition}
Let $\vc$ be a linear transformation from $M_{p,q}(\R)$ to $\R^{pq}$ that assigns to each $A \in M_{p,q}(\R)$ a vector $\vc(A) \in \R^{pq}$ with $(kp+r)$--th element of $\vc(A)$ equal to $a_{r\, k+1}$, $k=0,1,\ldots,q-1$ and $r \leq p.$ Transformation $\vc$ takes the columns of the matrix $A$ and puts them in a single column, starting with column $1.$

We may also define the inverse transformation $\mathrm{unvec}$ from $\R^{pq}$ to $M_{p,q}(\R).$ 
\end{definition}

Using operation $\vc$ we can transform a bilinear system to a linear system with an additional condition. 
\begin{theorem}
The set of solutions of a BLS
\begin{equation}\label{BLSthm3.1}
 y^T A_i x=g_i, \, i=1,\ldots, m,
\end{equation}
is equal to the set of solutions of the equation: 
\begin{equation*} 
 \mc A^T \vc K=g,
\end{equation*}
in which $K=yx^T$, and  $\mathcal{A}=(\vc A_1,\ldots,\vc A_m)$. 
\end{theorem}

\begin{proof}
Let us use standard notation $x_i$ for the $i$--th coordinate of the vector $x$ and $A_{i\, j}$ for the $(i,j)$--th element of the matrix $A.$ 
If we write a bilinear equation $y^TAx=g$ in coordinates we get:
\begin{equation}\label{GT1} 
\sum_{j=1}^p \sum_{k=1}^q A_{j \, k} y_jx_k=g.
\end{equation} 
If we define $K=yx^T,$ then we can write each equation $y^TA_ix=g$ in the following way: 
\begin{equation}\label{GT2} 
 (\vc A_i)^T \vc K=g.
\end{equation} 
Notice that we obtained a set of linear equations in variables $y_jx_k,$ $j=1,\ldots,p,$ $k=1,\ldots,q.$
Let $\mathcal{A}=(\vc A_1,\ldots,\vc A_m)$, and we can write the system as
\begin{equation}\label{GT3} 
 \mc A^T \vc K=g,
\end{equation}
 with $g=(g_1,\ldots,g_m)^T$. 
 
 Notice by the definition $K=yx^T$, $K$ can only have rank one or zero, hence not every solution to the linear system (\ref{GT3}) will give a solution to the original BLS.
\end{proof}

From our basic assumption on bilinear systems that the matrices $A_i$ are linearly independent, it follows that the rows of $\mc A^T$ are linearly independent,
As we know from the theory of linear systems every solution of a linear system can be written in the following way   
$$v_0+z_1 v_1+z_2 v_2+\ldots+z_r v_r,$$ 
in which $v_0$ is a solution to the system, and $v_1, \,v_2, \ldots, v_r$ form a basis for the solution space of the homogenous system $\mc A^T \vc K=0$ and $r=pq-m.$  
Let $K_s=\mathrm{unvec}( v_s)$ for $s=0,1,\ldots, r,$ and let $z=(z_1,\ldots, z_r).$ Then we define 
\begin{equation}\label{GT5}
K(z)=K_0+z_1K_1+z_2 K_2 +\ldots +z_r K_r.
\end{equation}
A nonzero solution to the linear system (\ref{GT3}) will give a solution to the bilinear system (\ref{BLSthm3.1}) if and only if the matrix $K(z)$ has rank one for some choice of $z.$
  
This method reduces a system of bilinear equations to a system of linear equations. Since the solvability of the system of linear equations is well understood, this method seems to simplify the problem. However, after solving a linear system we need to decide if an affine space (\ref{GT5}) contains a rank one matrix, which is in general a difficult problem.

When $\mc A^T$ has $pq$ rows it is invertible. This occurs when the system of bilinear equations has $pq$ equations. In this case, we have only one solution to the linear system, and we only need to check if the obtained solution gives a rank one matrix $K.$  Such bilinear systems are called  \emph{complete bilinear systems} and they have been treated in detail in \cite{MR2567143}. We will call bilinear systems of equations with fewer than $pq$ equations \emph{incomplete systems}.

An incomplete bilinear system can be completed to a complete bilinear system, by adding additional equations to the system. For the right hand sides of the added equations we put free parameters $z_i,$ $i=1,2,\ldots pg-m.$ In terms of the associated linear system this means
augmenting the $pq$-by-$m$ matrix $\mc A$ with linearly independent columns to a $pq$-by-$ pq$ invertible matrix $\mc B$ by adding extra columns to the right hand side of $\mc A,$ and extending the right hand side $g$ to $\npmatrix{g \\ z}.$ The solution to the linear system is then
\begin{equation}\label{KBgz}
\vc K=(\mc B^T)^{-1}\npmatrix{g \\ z}.
\end{equation}
As before, we need to find parameters $z_i$ that will give us a rank one matrix $K.$ Notice if we expand formula (\ref{KBgz}), we get a linear matrix polynomial in $z$:
$$K'(z)=K'_0+z_1 K'_1+\ldots+z_r K'_r.$$
Since the set of solutions of the incomplete system is just the union of the solution sets of all complete systems indexed by $z$, it suffice to find all rank one $K'(z)$'s. Since the $\vc K'(z)$ and $\vc K(z)$ solve the same linear system of equations, for each $K'(z)$ associated with a particular completion, we can always choose a basis $v_1,\ldots,v_r$ for the solution space of the homogenous system $\mc A^T \vc K=0$  such that $K'(z)=K(z)$. Therefore, the formula above is a equivalent to formula (\ref{GT5}). Formula (\ref{KBgz}) thus also serves as an alternate way to calculate the coefficients $K_i$ in formula (\ref{GT5}).

\section{Rank one Completions}\label{Section4}

In the previous section we have seen that the real difficulty in solving a bilinear system lies in finding a rank one matrix in an affine space defined by the equation:
 $$K(z)=K_0+z_1 K_1+\ldots+z_r K_r,$$
in which the matrices $K_i$ are obtained by solving a linear system of equations. As is well known, a nonzero matrix is rank one if and only if all its $2$-by-$2$ minors are equal to zero. The $2$-by-$2$ minors of $K(z)$ are quadratic functions in the $z_i$'s, hence solving a bilinear system comes down to solving a system of $\binom{p}{2}\binom{q}{2}$ quadratic equations in $pq-m$ variables.

\begin{example}\label{EX41}
Let us look at the bilinear system of equations defined by matrices $$A_1=\npmatrix{0 & 1& 0 \\ 0 & 0 & 1},\text{ }
 A_2=\npmatrix{-1 & 0& -1 \\ 0 & 0 & 0 } \text{ and }
 A_3=\npmatrix{0 & 1& 0 \\ -1 & 0 & 0 }$$
for an arbitrary right hand side $g=(g_1,g_2,g_3)^T.$ The corresponding linear equation is 
 $$\mc A^T \vc K=g,$$
with
 $$\mc A^T=\npmatrix{0 & 0& 1 & 0 & 0 & 1 \\
                                    -1 & 0& 0 & 0 & -1 & 0 \\
                                    0 & -1& 1 & 0 & 0 & 0}.$$
and 
 $$\vc K=\npmatrix{y_1x_1 & y_2x_1 & y_1x_2 & y_2 x_2 & y_1 x_3 & y_2 x_3}^T.$$                                     
Solving the linear system gives us the following solution for $\vc K:$
\begin{align*}
x_1y_1&=z_2 \\
x_1y_2&=-g_3+z_3 \\
x_2y_1&=z_3 \\
x_2y_2&=z_1 \\
x_3y_1&=-g_2-z_2 \\
x_3y_2&=g_1-z_3
\end{align*}
where $z_1,$ $z_2$ and $z_3$ are free parameters. To solve our bilinear system we need to find a rank one completion of 
 $$K=\npmatrix{z_2 & z_3 & -g_2-z_2 \\ 
                          -g_3+z_3 & z_1 & g_1-z_3}$$
for some choice of $z_1,$ $z_2,$ $z_3.$ Matrix $K$ will be rank one if and only if all its $2$-by-$2$ minors are equal to $0.$ This gives us a system of three quadratic equations:
\begin{align*}
z_1 z_2-z_3^2+z_3g_3&=0 \\
z_1z_2-z_3^2+z_1g_2+z_3g_1&=0 \\
z_2(g_1-g_3)+z_3g_2-g_2g_3&=0.
\end{align*}
If $g_1 \neq g_3$ then $K$ will be rank one, for example, for $z_1=0,$ $z_3=0,$ $z_2=\frac{g_2g_3}{g_1-g_3}.$ If $g_1=g_3,$ then $K$ will be rank one for $z_3=g_1,$ $z_1=0$ and arbitrary $z_2.$

This shows that bilinear system of equations defined by matrices $A_1,$ $A_2,$ $A_3$ is solvable for any right hand side $g.$ 
\end{example}

\begin{remark}
It is necessary to check all 2-by-2 minors to characterize all rank 1 points of $K(z)$.
Consider the 2-by-3 example:
\[
K(z)=\pmx{1&z&-z\\z&z&1}.
\]
Suppose the minor that is not contiguous is not checked.
The two checked minors are $z-z^2=0$ and $z+z^2=0$, whose common solution is $z=0$. Notice that the result is not rank one:
\[
\pmx{1&0&0\\0&0&1}.
\]
In fact, this 2-by-3 example can be extended to a matrix function of any size by adding zeros in all additional entries, creating counter--examples for $K(z)$ of any size.
\end{remark}

\begin{remark}[Special case when $r=1$]
In case $r=1,$ only one $z$ will appear in the $K$ matrix, which will be of the form $K(z)=G+zH.$
In this event, necessary for solvability of the bilinear systems (i.e. rank of $K(z)$ is equal to one) is that $|\mathrm{rank}G- \mathrm{rank}H|\leq 1.$ Moreover, the existence of a rank $1$ matrix $K(z)$ may easily be assessed via the vanishing $2$-by-$2$ minors approach. Minors in which no $z$ appears must already be $0$ and those in which $z$ does appear must all have a common root. Since these roots may be calculated exactly via the quadratic formula, this requirement may be assessed precisely in polynomial time. 
 
Even when $r$ is small, $z_j$'s may appear in many positions. If $z_j$'s appear in only a few lines of $K,$ then there is a stringent rank condition on the data that is necessary for solvability. If, for example, $z_j$'s appear in only $s$ rows of $K$ ($s$ columns of $K$), the remaining $p-s$ rows ($q-s$ columns) must already have rank $1$.
\end{remark}

\section{Solvability of Bilinear Systems for all Right Hand Sides}\label{Section5}

Now, we suppose that $A_1,\ldots, A_m$ are fixed (and linearly independent) and ask under what conditions the bilinear system is solvable for all right hand sides $g.$ We refer to this situation, which depends only upon $m$ and $A_1, \ldots, A_m$, as \emph{always solvable}.  As we shall see, this can only happen for certain values of $m$.

It is clear that when $m=1,$ every bilinear system is solvable because $A_1 \neq 0$ due to our linear independence hypothesis. Interestingly this remains true for $m=2$ (and we know already that it is not so for $m=3$, as we saw in Example (\ref{Ex21})) 

\begin{theorem}
Under the linear independence hypothesis, every bilinear system with $m \leq 2$ over any field $\F$ is solvable. 
\end{theorem} 

\begin{proof}
Let $A_1$ and $A_2$ be linearly independent matrices. We want to show that 
then the bilinear system  \begin{equation}\label{twoBLS} 
 y^TA_1x=g_1, \, y^TA_2x=g_2
 \end{equation} 
is solvable for all right hand sides $(g_1,g_2).$ We have already noted that a homogeneous system always has a trivial solution $x=0$ or $y=0,$ so we may assume that $g_1$ and $g_2$ are not both equal to zero. Linear independence (dependence) of matrices $A_1$ and $A_2$ is preserved after we apply an invertible linear transformation that takes $g=(g_1,g_2)^T$ to $(1,0)^T,$ so it is sufficient to consider the system 
 \begin{equation}\label{twoBLS1} 
 y^TA_1x=1, \, y^TA_2x=0.
\end{equation}  
Clearly, the system (\ref{twoBLS1}) is solvable unless the dimension of the span of vectors $A_1x$ and $A_2x$ is the same as the dimension of the span of a vector $A_2x$ for all $x.$ Let us assume that this condition h olds. In particular, this implies that the null space of $A_2$ is contained in the null space  of $A_1.$ 

Let $b_1,\ldots,b_s$ be a basis of the null space of $A_2,$ and let us complete this basis with $b_{s+1},\ldots, b_q$ to a basis for $\mathbb{R}^q.$ Then, there exist scalars $\alpha_k \in \F$ such that:
 $$A_1b_k=\alpha_k A_2b_k, \,k=s+1,\ldots,q,$$ by our assumption. If the $\alpha_k$ are all equal then $A_1=\alpha_{s+1} A_2.$ If not, we may assume that $\alpha_{s+1} \neq \alpha_{s+2}.$ Then $$A_1(b_{s+1}+b_{s+2})=\alpha A_2(b_{s+1}+b_{s+2}).$$ On the other hand $$A_1(b_{s+1}+b_{s+2})=\alpha_{s+1} A_2 b_{s+1}+\alpha_{s+2} A_2 b_{s+2}.$$ Then $(\alpha_{s+1}-\alpha)A_2 b_{s+1}+(\alpha_{s+2}-\alpha)A_2 b_{s+2}=0.$ Since $b_{s+1}$, $b_{s+2}$ and the basis elements  $b_1,\ldots,b_s$ of the null space of $A_2$ are linearly independent
this implies $(\alpha_{s+1}-\alpha)b_{s+1}+(\alpha_{s+2}-\alpha)b_{s+2}=0$, and this contradicts our assumption that $b_{s+1}$ and $b_{s+2}$ are linearly independent.  
\end{proof}

\begin{theorem}\label{ASbound}
Now let $\F$ be either $\R$, $\C$ or a finite field, and let linearly independent $A_1,\ldots, A_m \in M_{p,q}(\F)$ define a bilinear system. If this bilinear system is always solvable, then $m \leq p+q-1$.
\end{theorem}

\begin{proof}
Assume  $m \ge p+q$.
Define a degree-two polynomial map on $\F^q\times\F^p$ to $\F^m$ by:
\begin{eqnarray*}
F: &\F^q\times\F^p &\longrightarrow \F^m \\
& (x,y)&\longmapsto (y^TA_1x,\ldots, y^TA_mx) 
\end{eqnarray*}
It suffices to prove that the image of $F$ is strictly contained in $\F^m$. We do so respectively for  the cases $\F=\R$, $\F=\C$ and $|\F|<\infty$.

\begin{case}[$\F=\R$]
Function $F$ is $C^\infty(\R^{p+q})$ smooth. By Sard's Lemma, the subset of $\R^{p+q}$ where $\rk (dF)<m$ has an image of measure 0 in $\R^m$. Therefore, it suffices to show $\rk (dF)<m$ for all pairs $(x,y)$.

Denote $y^TA_ix=y_j A_{ijk}x_k$ ($\sum_{j=1}^p\sum_{k=1}^q$ omitted according to Einstein's notation). It follows that
\[
dF=\pmx{y_j A_{1j1} &\ldots & y_j A_{1jq} & A_{11k}x_k&\ldots&A_{1pk}x_k\\
y_j A_{2j1} &\ldots & y_j A_{2jq} & A_{21k}x_k&\ldots&A_{2pk}x_k\\
\vdots&\vdots&\vdots &\vdots&\vdots&\vdots\\
y_j A_{mj1} &\ldots & y_j A_{mjq} & A_{m1k}x_k&\ldots&A_{mpk}x_k\\}=\pmx{y^TA_1 & x^TA_1^T\\\vdots&\vdots\\y^TA_m & x^TA_m^T}.
\]
Since
\[
dF \pmx{x\\-y}=\pmx{y^TA_1 & x^TA_1^T\\\vdots&\vdots\\y^TA_m & x^TA_m^T}\pmx{x\\-y}=\pmx{y^TA_1x- x^TA_1^Ty\\\vdots&\\ y^TA_mx- x^TA_m^Ty}= \pmx{0\\\vdots&\\ 0},
\]
the columns of matrix $dF$ are linearly dependent. Therefore, $\rk(dF)<m$ holds for any $(x, y)\neq0$.
For the $(x, y)=(0,0)$ case, $dF=(0)$ and $\rk(dF)=0$. Therefore, $F(\F^q\times\F^p)\subset\F^m$.
\end{case}
\begin{case}[$\F=\C$]
Function $F$ can be viewed as a function from $\R^{2p+2q}$ to $\R^{2m}$ (denoted as $\widetilde F$):
\begin{eqnarray*}
\widetilde F: &\R^{2p+2q} &\longrightarrow \R^{2m} \\
& (\textnormal{Re}  (x), \textnormal{Im} (x),\textnormal{Re} (y),\textnormal{Im} (y))&\longmapsto (\textnormal{Re}  (F(x,y)), \textnormal{Im} (F(y,x))) 
\end{eqnarray*}

Since $\widetilde F$ is still a polynomial, it is $C^\infty(\R^{2p+2q})$ smooth. By Sard's Lemma, the subset of $\R^{2p+2q}$ where $\rk (dF)<2m$ has an image of measure 0 in $\R^{2m}$. Therefore, it suffices to show that $\rk (dF)<2m$ for all pairs $(x,y)$. The proof is as follows:

Denote $a:=\textnormal{Re}  (x), b:=\textnormal{Im} (x),c:=\textnormal{Re} (y), d:=\textnormal{Im} (y),B_i:=\textnormal{Re}  (A_i)$, and $C_i:=\textnormal{Im} (A_i)$. By definition
\begin{eqnarray*}
&\widetilde F&(a,b,c,d)=(\textnormal{Re}  (F(x,y)), \textnormal{Im} (F(y,x)))\\
&=& (\textnormal{Re}  ((c+Id)(B_i+IC_i)(a+Ib),\ldots,\textnormal{Im} ((c+Id)(B_i+IC_i)(a+Ib),\ldots)\\
&=& (c^TB_ia-d^TC_ia-c^TC_ib-d^TB_ib,\ldots,-d^TC_ib+d^TB_ia+c^TC_ia+c^TB_ib,\ldots).
\end{eqnarray*}
Therefore,
\begin{eqnarray*}
d\widetilde F& =& \pmx{\displaystyle\pde{\textnormal{Re} (F)}{a}&\displaystyle\pde{\textnormal{Re}  (F)}{b}&\displaystyle\pde{\textnormal{Re}  (F)}{c}&\displaystyle\pde{\textnormal{Re}  (F)}{d}\\ \displaystyle\pde{\textnormal{Im} (F)}{a}&\displaystyle\pde{\textnormal{Im} (F)}{b}&\displaystyle\pde{\textnormal{Im} (F)}{c}&\displaystyle\pde{\textnormal{Im} (F)}{d}}\\
&=& \pmx{c^TB_1-d^TC_1&a^TB_1^T-b^TC_1^T&-c^TC_1-d^TB_1&-a^TC_1^T-b^TB_1^T\\ \vdots&\vdots&\vdots&\vdots\\c^TB_1-d^TC_m&a^TB_m^T-b^TC_m^T&-c^TC_m-d^TB_m&-a^TC_m^T-b^TB_m^T\\ d^TB_1+c^TC_1&a^TC_1^T+b^TB_1^T&-d^TC_1+c^TB_1&-b^TC_1^T+a^TB_1^T\\ \vdots&\vdots&\vdots&\vdots \\d^TB_m+c^TC_m&a^TC_m^T+b^TB_m^T&-d^TC_m+c^TB_m&-b^TC_m^T+a^TB_m^T}.
\end{eqnarray*}
Since
\[
d\widetilde F \pmx{a\\-c\\b\\-d}=\pmx{0\\\vdots\\0},
\]
the columns of matrix $dF$ are linearly dependent. Therefore, for any $(a,b,c,d)\neq0$, $\rk(d \widetilde F)<2m$.
For $(a,b,c,d)=(0,0,0,0)$ case, $d\widetilde F=(0)$, $\rk(d\widetilde F)=0$. Hence $F(\F^q\times\F^p)\subset\F^m$.
\end{case}
\begin{case}[$|\F|=N<\infty$]
It suffices to count the cardinality of $F(\F^q\times\F^p)$ and $\F^m$.  Define relation $(cx,y)\sim(x,cy)$ on ${\F^q}\backslash\{0\}\times{\F^p}\backslash\{0\}$. Notice $F(cx,y)=F(x,cy)$ $(c\in\F^*)$, so that the mapping:
\begin{eqnarray*}
\widetilde F: &{\F^q}\backslash\{0\}\times{\F^p}\backslash\{0\}/\sim &\longrightarrow \F^m \\
& [(x,y)]&\longmapsto F(x,y)
\end{eqnarray*}
is well defined. Also notice
\begin{equation*}
F(\{0\}\times\F^p\cup\F^q\times\{0\})=\{0\},
\end{equation*}
so that we may estimate
\begin{equation*}
\begin{split}
|F(\F^q\times\F^p)|&=
|F({\F^q}\backslash\{0\}\times{\F^p}\backslash\{0\})|+1\\
&\le\frac{(N^q-1)(N^p-1)}{N-1}+1\le(N^q-1)(N^p-1)+1\\
&<N^{p+q}\le N^m=|\F^m|
\end{split}
\end{equation*}
Hence $F(\F^q\times\F^p)\subset\F^m$.
(We assumed $|\F|>1$.)
\end{case}
\end{proof}


When $p=1$ or $q=1,$ the bound in Theorem \ref{ASbound} may be attained. In fact, under the linear independence hypothesis on the matrices $A_i$, every bilinear system is always solvable. 

When both $p,q>1,$ it is not difficult to construct $A_i$'s for which the bilinear system is always solvable whenever $m \leq p+q-2.$ 

\begin{example}
Consider a bilinear system defined by matrices $A_i=E_{1\,i},i=1,\ldots, q-1$ and 
$A_{q+i}=E_{i\, n},$ $i=1,\ldots, p-1.$ ($E_{ij}$ denotes a matrix whose $(i,j)$-th element is equal to $1$ and all other elements are equal to zero.) Then
 $$K(z)=\npmatrix{g_1 & g_2 & \ldots & g_{q-1} & z_1 \\
                               z_2 & z_3 &  \ldots & z_q & g_q \\
                                \vdots      & &\ddots & \vdots &\\
                                 z_{r-q+2}    & & \ldots & z_r& g_{p+q-2}},$$
                                      which may be completed to a rank one matrix, whatever the $g.$ (This construction is not unique.)
\end{example}

When $p,q>1$ and $m=p+q-1,$ we may construct matrices $A_i$ for which the bilinear system is  solvable for "almost" all $g$'s with few exceptions.  
 
 \begin{example}
 A bilinear system defined by matrices $E_{11}, E_{12}, \ldots, E_{1q}$ and matrices 
 $E_{2q}, E_{3q}, \ldots, E_{pq}$ is solvable whenever $g_q \neq 0$ and has no solution when $q_g=0$ and either $g_i=0$ for some $i < q$ or for some $i>q.$ Again looking at the matrix $K(z)$ makes this observation clear.  
 \end{example}
 
 We do not know of an always solvable bilinear system when $p,q >1$ and $m=p+q-1.$ We conjecture that they do not exist. 
 
 In general, then, whether a bilinear system is always solvable for $m \leq p+q-1$ depends upon the data $A_1,\ldots, A_m$, and when $m \leq 2,$ every bilinear system is always solvable. Even when $p$ and $q$ are large and $m=3,$ a bilinear system may not be always solvalbe for "local reasons". For example, in Example \ref{Ex21} we can choose $g_1$, $g_2$, $g_3$ so that the BLS is not solvable, and then embed it in a larger example by adding equations, with many more $x$ and $y$ variables. The resulting BLS will still not have a solution. This also shows that a bilinear system may not be always solvable for any $m,$ $3 \leq m \leq p+q-1.$ 
 
How, then, may the data be tested for always solvability when $m$ is of an appropriate size? We mention several ideas. First, note that a bilinear system becomes a linear system when enough of the variables are taken to have particular values. If, because of the structure of the bilinear system, the result is a linear system with $m$ linearly independent equations, then the obtained linear system has a solution for every right hand side and so does the bilinear system, which is then always solvable. Furthermore, there will be a solution in which some of the variables are constant (independent of $g$). As we will note, it may happen that a bilinear system is always solvable without this happening. 

\begin{example}
If in the bilinear system 
 \begin{align*}
 x_1y_3 + x_2y_2&=g_1 \\
 x_2y_1+x_2y_3&=g_2 \\
 x_2y_2&=g_3
 \end{align*}  
we take $x_2=1$ and $y_3=1,$ the linear system 
\begin{align*}
x_1&=g_1-1 \\
y_1&=g_2-1 \\
y_2&=g_3
\end{align*}
results. Since the latter has the invertible coefficient matrix $I_3,$ it has a solution for all right hand sides and the given bilinear system is always solvable with solution 
 $$x=\npmatrix{g_1-1 \\1} \text{ and }y=\npmatrix{g_2-1 \\g_3 \\ 1}.$$ 
\end{example}

In some cases we can specify variables of one type; we specify only $x$ or only $y$ variables to get a always solvable system.  If all the $y$ ($x$) variables are specified and $Y$ has full row rank ($X$ has full column rank) for some choice of $y$ ($x$), the bilinear system is always solvable. Of course, for this to happen we must have $m \leq p$ ($q$). 

Another way in which a bilinear system is always solvable is based upon the arrangement of dyads $x_iy_j$ that appear or, equivalently, the collective support of the matrices $A_i.$ Let $\mc P,$ $p$-by-$q$ be this collective support. We say that the bilinear system satisfies the \emph{$3$--corner porperty} if $\mc P$ has no $2$-by-$2$ subpattern with $3$ out of $4$ entries nonzero. Note that this happens if, up to permutation equivalence, $\mc P$ is contained in a pattern of the form: (all blank entries are zero)
  $$\npmatrix{*\cdots * & 0 & & & \\
                   & * & & &  \\
                   & \vdots & & &  \\
                   & * & & &  \\
                   & 0 & * \cdots * & 0 & \\
                   & & & *  &\\
                   & & & \vdots & \\
                   & & & * & \\
                   & & & 0 & * \cdots *\\},$$
\emph{a deleted echelon form}, i.e. a pattern in which the nonzero entries in a row do not intersect nonzero entries in a column, unless the intersecting entry is unique nonzero entry in a least one of the row or the column.

\begin{theorem}\label{DEF}
A bilinear system is always solvable if it satisfies the $3$--corner property (or it is equivalent to a bilinear system that satisfies the  $3$--corner property).
\end{theorem}

\begin{proof}
Let a bilinear system satisfy the $3$--corner property and let $P$ a $(0-1)$--matrix whose support is equal to the collective support of the bilinear system. Then $y^TPx$ is a sum of the 
terms of the form:
$$y_i(\sum_{j=k_i}^{l_i} x_j) \text{ and } x_j(\sum_{i=s_j}^{t_j} y_i),$$
where each $y_i$ and each $x_j$ appear at most once. If $y_i$ appears in the term of the form $y_i(\sum_{j=k_i}^{l_i} x_j)$, we specify $y_i=1$ and if $x_j$ appears in the term of the form $x_j(\sum_{i=s_j}^{t_j} y_i)$, we specify $x_j=1$. After this specification, each matrix in a bilinear system gives a linear equation in the remaining unspecified variables. Since the matrices are linearly independent, a nonsingular linear system results for the remaining variables. 

Recall that equivalence doesn't affect the solvability of a bilinear system, so that the condition can be weakened to a system being only equivalent to a bilinear system that satisfies the $3$--corner property.
\end{proof}

Note that the number of positions in a deleted echelon form is at most $p+q-2$ if $p,q >1,$ and it is actually $p+q$ less the minimum number of lines that cover it, if there are no zero lines. In particular, if $p=1$($q=1$), it will be $p+q-1,$ but this is the only case it can be that high. 

Note that a bilinear system may be always solvable without any of the above sufficient conditions occurring. 

\begin{example}
Return to the Example \ref{EX41}. In this case the $3$--corner property does not hold as the support is 
 $$\npmatrix{* & * & * \\ * & 0 & *},$$
 while rank $X \leq 2$ and 
  $$Y=\npmatrix{0 & y_1& y_2 \\ -y_1 & 0  &-y_1 \\ -y_2 & y_1& 0}$$ 
  is identically singular. For the bilinear forms
   \begin{align*}
   x_1y_2+x_2y_3 \\
   x_1y_1+x_1y_3\\
   x_1y_2-x_2y_1,
   \end{align*}
   no specification of just one variable gives a linear system and no specialization of two variables that gives a linear system gives an invertible one. However, the analysis of Section 4 shows that there is a solution for every $g.$ 
\end{example}

\bibliographystyle{plain}
\bibliography{Bilinear}

\begin{thebibliography}{1}

\bibitem{892611}
Scott Cohen and Carlo Tomasi.
\newblock Systems of bilinear equations.
\newblock Technical report, Stanford, CA, USA, 1997.

\bibitem{Courtois02cryptanalysisof}
Nicolas~T. Courtois and Josef Pieprzyk.
\newblock Cryptanalysis of block ciphers with overdefined systems of equations.
\newblock pages 267--287. Springer, 2002.

\bibitem{MR2567143}
Charles~R. Johnson and Joshua~A. Link.
\newblock Solution theory for complete bilinear systems of equations.
\newblock {\em Numer. Linear Algebra Appl.}, 16(11-12):929--934, 2009.

\bibitem{MR2202432}
Charles~R. Johnson and Maria da~Gra{\c{c}}a Marques.
\newblock Patterns of commutativity: the commutant of the full pattern.
\newblock {\em Electron. J. Linear Algebra}, 14:43--50 (electronic), 2005.

\bibitem{bobpalais}
Bob Palais.
\newblock Quaternions in three dimensions.
\newblock {\em arXiv:1011.6279v1}, 2010.

\bibitem{Wolf}
Christopher Wolf.
\newblock {\em Multivariate quadratic polynomials in public key cryptography}.
\newblock PhD thesis, Katholieke Universiteit Leuven, 2005.

\bibitem{Xie:2009:QEK:1698100.1698124}
Jia Xie, Weiwei Cao, and Tianze Wang.
\newblock Information security applications.
\newblock chapter Quadratic Equations from a Kind of S-boxes, pages 239--253.
  Springer-Verlag, Berlin, Heidelberg, 2009.

\end{thebibliography}

\end{document}